\documentclass[11pt,reqno,oneside]{amsart}
 \usepackage{latexsym}
 \usepackage{amssymb}
 \usepackage{amsfonts}
\usepackage{graphics}
\usepackage{mathcomp}
 \author[Radoux]{F. Radoux}
\thanks{University of Li\`ege, Institute of mathematics, Grande Traverse, 12 - B37, B-4000
 Li\`ege, Belgium email : Fabian.Radoux@ulg.ac.be}
\date{\today} 
\title[Explicit formula]{An explicit formula for the natural and conformally invariant quantization}

\newtheorem{lem}{Lemma}
\newtheorem{thm}[lem]{Theorem}
\newtheorem{prop}[lem]{Proposition}

\theoremstyle{remark}

\theoremstyle{definition}
\newtheorem{defi}{Definition}

\newcommand{\R}{\mathbb{R}}

\newcommand{\N}{\mathbb{N}}

\newcommand{\g}{\mathfrak{g}}

\newcommand{\tree}[1][\gamma]{\mathcal{T}_{{#1}}}

\newcommand{\gm}{{\g_{-1}}}
\newcommand{\Cf}{C^\flat}

\begin{document}
\begin{abstract}
In \cite{Leconj}, P. Lecomte conjectured the existence of a natural and conformally invariant quantization. In \cite{MR}, we gave a proof of this theorem thanks to the theory of Cartan connections. In this paper, we give an explicit formula for the natural and conformally invariant quantization of trace-free symbols thanks to the method used in \cite{MR} and to tools already used in \cite{moi} in the projective setting. This formula is extremely similar to the one giving the natural and projectively invariant quantization in \cite{moi}. 
\end{abstract}
%%%
\maketitle
%%%
\section{Introduction}
A quantization can be defined as a linear bijection from the space $\mathcal{S}(M)$ of symmetric contravariant tensor fields on a manifold $M$ (also called the space of \emph{Symbols}) to the space $\mathcal{D}_{\frac{1}{2}}(M)$ of
differential operators acting between half-densities.

It is known that there is no natural quantization procedure. In other words, the spaces of symbols and of differential operators are not isomorphic as representations of $\mathrm{Diff}(M)$.

The idea of equivariant quantization, introduced by P. Lecomte and V. Ovsienko in \cite{LO} is to reduce the group of local diffeomorphisms in the following way : if a Lie group $G$ acts (locally) on a manifold $M$, the action can be lifted to tensor fields and to differential operators and symbols. A $G$-equivariant quantization is then a quantization that exchanges the actions of $G$ on symbols and differential operators.

In \cite{DLO}, the authors considered the group $SO(p+1,q+1)$ acting on the space $\R^{p+q}$ or on a manifold endowed with a flat pseudo-conformal structure of signature $(p,q)$. They showed the existence and uniqueness of a  $SO(p+1,q+1)$-equivariant quantization in non-critical situations.

The problem of the $so(p+1,q+1)$-equivariant quantization on $\R^{m}$ has a counterpart on an arbitrary manifold $M$. In \cite{Leconj}, P. Lecomte conjectured the existence of a natural and conformally invariant quantization, i.e. a quantization procedure depending on a pseudo-riemannian metric, that would be natural (in all arguments) and that would be left invariant by a conformal change of metric.

We proved in \cite{MR} the existence of such a quantization using Cartan connections theory.

The goal of this paper is to obtain an explicit formula on $M$ for the natural
and conformally invariant quantization of trace-free symbols. This task can be
realized using tools exposed in \cite{MR} and \cite{Slo1}.

The paper is organized as follows. In the first section, we recall briefly the notions exposed in \cite{MR} necessary to
understand the article. In the second part, we calculate the explicit formula giving the natural and conformally invariant quantization for trace-free symbols on the Cartan fiber bundle using the method exposed in \cite{MR}. In the third section, we develop as in \cite{moi} this formula in terms of natural operators on the base manifold $M$, using tools explained in \cite{Slo1}, in order to obtain the announced explicit formula. It constitutes the generalization to any weight of density of the formula given by M. Eastwood in \cite{East} thanks to a completely different method. Moreover, Eastwood formula is given under a different form. 
%The paper is organized as follows. We define the general concept of conformally
%invariant and natural quantizations over arbitrary manifolds in section \ref{tens}. 
%In section \ref{curvtools}, we adapt the tools of \cite{MR1} to the conformal situation, and
%we eventually use these tools in order to build the quantization in section \ref{curvconst}. 
%%%%
%%%%
\section{Fundamental tools}\label{tens}
Throughout this work, we let $M$ be a smooth manifold of dimension $m\geq 3$.
\subsection{Tensor densities}
Denote by $\Delta^{\lambda}(\R^m)$ the one dimensional representation of $GL(m,\R)$ given by
\[\rho(A) e = \vert det A\vert^{-\lambda} e,\quad\forall A\in
GL(m,\R),\;\forall e\in \Delta^{\lambda}(\R^m).\]
The vector bundle of $\lambda$-densities is then defined by
\[P^1M\times_{\rho}\Delta^{\lambda}(\R^m)\to M,\]
where $P^1M$ is the linear frame bundle of $M$.

Recall that the space $\mathcal{F}_{\lambda}(M)$ of smooth sections of this bundle, the space of $\lambda$-densities,  can be identified with the space $C^{\infty}(P^1M,
\Delta^{\lambda}(\R^m))_{GL(m,\R)}$ of functions $f$ such that
\[f(u A) = \rho (A^{-1}) f(u)\quad \forall u \in P^1M,\;\forall A\in
GL(m,\R).\]
%%%%%%%%
\subsection{Differential operators and symbols}
%%%%%%%
As usual, we denote by $\mathcal{D}_{\lambda,\mu}(M)$ the space 
of differential operators from $\mathcal{F}_{\lambda}(M)$ to
$\mathcal{F}_{\mu}(M)$.

The space $\mathcal{D}_{\lambda,\mu}$ is filtered by 
the order of differential operators.
The space of \emph{symbols} is then the associated graded space of $\mathcal{D}_{\lambda,\mu}$. It is also known that the principal operators $\sigma_l$ $(l\in \N)$ allow to identify the space of symbols with the space of contravariant symmetric tensor fields with coefficients in $\delta$-densities where $\delta=\mu-\lambda$ is the shift value.

More precisely, we denote by $S^l_{\delta}(\R^m)$ or simply $S^l_{\delta}$ the space $S^l\R^m\otimes\Delta^{\delta}(\R^m)$ endowed with the natural representation $\rho$ of $GL(m,\R)$. Then the vector bundle of symbols of degree $l$ is 
\[P^1M\times_{\rho}S^l_{\delta}(\R^m)\to M.\]
The space $\mathcal{S}^l_{\delta}(M)$ of symbols of degree $l$ is then the space of smooth sections of this bundle, which can be identified with $C^{\infty}(P^1M, S^l_{\delta}(\R^m))_{ GL(m,\R)}$.
Finally, the whole space of symbols is
\[\mathcal{S}_{\delta}(M)
=\bigoplus_{l=0}^{\infty}\mathcal{S}^l_{\delta}(M),\]
endowed with the classical actions of diffeomorphisms and of vector fields.
\subsection{Natural and equivariant quantizations}
A \emph{quantization on $M$} is a linear bijection $Q_{M}$
  from the space of symbols $\mathcal{S}_{\delta}(M)$ to the space of differential operators $\mathcal{D}_{\lambda,\mu}(M)$ such
  that
\[\sigma_l(Q_{M}(S))=S,\;\forall S\in\mathcal{S}_{\delta}^{l}(M),\;\forall
  l\in\mathbb{N},\]
where $\sigma_l$ is the principal symbol operator on the space of operators of order less or equal to $l$.

In the conformal sense, a \emph{natural quantization} is a collection of quantizations $Q_M$ depending on a pseudo-Riemannian metric such that 
\begin{itemize}
\item
For all pseudo-Riemannian metric $g$ on $M$, $Q_{M}(g)$ is a quantization,
\item
If $\phi$ is a local diffeomorphism from $M$ to $N$, then one has
\[Q_{M}(\phi^{*}g)(\phi^{*}S)=\phi^{*}(Q_{N}(g)(S)),\]
 for all pseudo-Riemannian metrics $g$ on $N$, and all
$S\in\mathcal{S}_{\delta}(N).$
\end{itemize}
Recall now that two pseudo Riemannian metrics $g$ and $g'$ on a manifold $M$ are conformally equivalent if and only if there exists a positive function $f$ such that $g'=fg$. 

A quantization $Q_M$ is then \emph{conformally equivariant} if one has $Q_M(g)=Q_M(g')$  whenever $g$ and $g'$ are conformally equivalent.
%%%%
\subsection{Conformal group and conformal algebra}
These tools were presented in detail in \cite[Section 3]{MR}. We give here the
 most important ones for this paper to be self-contained.

%%%%%
Given $p$ and $q$ such that $p+q=m$, we consider the conformal group $G=SO(p+1,q+1)$ and its following subgroup $H$:
\[H=\{ \left( 
           \begin{array}{ccc}
              a^{-1} & 0 & 0\\
              a^{-1}A\xi^\sharp & A & 0\\
              \frac{1}{2a}|\xi|^2 & \xi & a
           \end{array}
        \right): A\in
        O(p,q),a\in\mathbb{R}_{0},\xi\in\mathbb{R}^{m*}\}/\{\pm I_{m+2}\}.\]
The subgroup $H$ is a semi-direct product $G_{0}\rtimes G_{1}$. Here $G_{0}$ is isomorphic to $CO(p,q)$ and $G_{1}$ is isomorphic to $\mathbb{R}^{m*}$. 
        
The Lie algebra of $G$ is $\mathfrak{g}=so(p+1,q+1)$. It decomposes as a direct sum of subalgebras :
\begin{equation}\label{grading}\g=\mathfrak{g}_{-1}\oplus\mathfrak{g}_{0}\oplus\mathfrak{g}_{1}\end{equation}
where $\g_{-1}\cong\R^{m}$, $\g_0\cong co(p,q)$, and $\g_1\cong\mathbb{R}^{m*}$.

This correspondence induces a structure of Lie algebra on $\mathbb{R}^{m}\oplus
co(p,q)\oplus\mathbb{R}^{m*}$. It is easy to see that the adjoint actions of $G_0$ and of $co(p,q)$ on $\gm=\R^m$ and on $\g_1={\R^m}^*$ coincides with the natural actions of $CO(p,q)$ and of $co(p,q)$. It is interesting for the sequel to note that :
$$[v,\xi]=v\otimes\xi+\xi(v)I_{m}-\xi^{\sharp}\otimes v^{\flat},$$
if $v\in\g_{-1}$, $\xi\in\g_{1}$ and if $I_{m}$ denotes the identity matrix of dimension $m$. The applications $\flat$ and $\sharp$ represent the classical isomorphisms between $\R^{m}$ and $\R^{m*}$ detailed in \cite{MR}.

The Lie algebras corresponding to $G_{0}$,
$G_{1}$ and $H$ are respectively $\mathfrak{g}_{0}$, $\mathfrak{g}_{1}$, and
$\mathfrak{g}_{0}\oplus\mathfrak{g}_{1}$.
%%%%%
\subsection{Cartan fiber bundles}
It is well-known that there is a bijective and natural correspondence between
the conformal structures on $M$ and the reductions of $P^{1}M$ to the
structure group $G_{0}\cong CO(p,q)$. The representations $(V,\rho)$ of $GL(m,\R)$ defined so far can be restricted to the group $CO(p,q)$. Therefore, once a conformal structure is given, i.e. a reduction $P_0$ of $P^1M$ to $G_0$, we can identify tensors fields of type $V$ as $G_0$-equivariant functions on $P_0$.

In \cite{Koba}, one shows that it is possible to
associate at each $G_{0}$-structure $P_{0}$ a principal $H$-bundle $P$ on $M$,
this association being natural and obviously conformally invariant. Since $H$ can be considered as a subgroup of $G^2_m$, this $H$-bundle can
be considered as a reduction of $P^{2}M$. The relationship between conformal structures and reductions of $P^2M$ to $H$ is given by the following proposition.
\begin{prop}
There is a natural one-to-one correspondence between the conformal equivalence
classes of pseudo-Riemannian metrics on $M$ and the reductions of $P^2M$ to $H$.\end{prop}
Throughout this work, we will freely identify conformal structures and reductions of $P^2M$ to $H$. 
\subsection{Cartan connections}
Let $L$ be a Lie group and $L_{0}$ a closed subgroup. Denote by $\mathfrak{l}$
and $\mathfrak{l}_{0}$ the corresponding Lie algebras. Let $N\to M$ be a
principal $L_{0}$-bundle over $M$, such that $\dim M=\dim L/L_{0}$. A Cartan
connection on $N$ is an $\mathfrak{l}$-valued one-form $\omega$ on $N$ such
that
\begin{enumerate}
\item
If $R_{a}$ denotes the right action of $a\in L_0$ on $N$, then
$R_{a}^{*}\omega=Ad(a^{-1})\omega$,
\item
If $k^{*}$ is the vertical vector field associated to $k\in\mathfrak{l}_0$,
then $\omega(k^{*})=k,$
\item
$\forall u\in N$, $\omega_{u}:T_{u}N\mapsto\mathfrak{l}$ is a linear
  bijection.
\end{enumerate}
When considering in this definition a principal $H$-bundle $P$, and taking as group $L$ the group $G$ and for $L_0$ the group $H$, we obtain the definition of Cartan conformal connections.

If $\omega$ is a Cartan connection defined on an $H$-principal bundle $P$, then its
curvature $\Omega$ is defined by
 \begin{equation}\label{curv} 
\Omega = d\omega+\frac{1}{2}[\omega,\omega].
\end{equation}
The notion of \emph{Normal} Cartan connection is defined by natural conditions
imposed on the components of the curvature.

Now, the following result (\cite[p. 135]{Koba}) gives the relationship between conformal
structures and Cartan connections :
\begin{prop}
 A unique normal
 Cartan conformal connection is
 associated to every conformal structure $P$. This association is natural.
\end{prop}
The connection associated to a conformal structure $P$ is called the normal conformal connection of the conformal structure.

\subsection{Lift of equivariant functions}\label{Lift}

In a previous subsection, we recalled how to associate an $H$-principal bundle $P$ to a conformal structure $P_0$. We now recall how the densities and the symbols can be regarded as equivariant functions on $P$.

If $(V,\rho)$ is a representation of $G_0$, then we may extend it to a representation $(V,\rho')$ of $H$ (see \cite{MR}). Now, using the representation $\rho'$, we can recall the relationship between
equivariant functions on $P_{0}$ and equivariant functions on $P$ (see \cite{Slo1}): if we
denote by $p$ the projection $P\to P_{0}$ , we have
\begin{prop}
If $(V,\rho)$ is a representation of $G_0$, then the map
$$p^{*}:C^{\infty}(P_{0},V)\mapsto C^{\infty}(P,V):f\mapsto f\circ p$$
defines a bijection from $C^{\infty}(P_{0},V)_{G_0}$ to
$C^{\infty}(P,V)_{H}$.
\end{prop}

As we continue, we will use the representation $\rho'_*$ of the Lie algebra of
$H$ on $V$. If we recall that this algebra is isomorphic to
 $\g_{0}\oplus\g_{1}$, then we have 
\begin{equation}\label{rho}\rho'_* (A, \xi) = \rho_*(A),\quad\forall A\in\g_{0}, \xi\in \g_{1}.\end{equation}

\subsection{The application $Q_{\omega}$}
The construction of the application $Q_{\omega}$ is based
on the concept of invariant differentiation developed in
\cite{Slo1}.
Let us recall the definition :

\begin{defi}
If $f\in C^{\infty}(P,V)$ then $(\nabla^{\omega})^k f
\in C^{\infty}(P,\otimes^k\R^{m*}\otimes V)$ is defined by
\[(\nabla^{\omega})^k f(u)(X_1,\ldots,X_k) = L_{\omega^{-1}(X_{1})}\circ\ldots\circ L_{\omega^{-1}(X_{k})}f(u)\]
for $X_1,\ldots,X_k\in\R^m$.
\end{defi}

\begin{defi}The map $Q_{\omega}$ is defined by its
  restrictions to $C^{\infty}(P,\otimes^k\g_{-1}\otimes\Delta^{\delta}(\R^{m}))$, $(k\in \N)$ : we set
\begin{equation}\label{Qom}
Q_{\omega}(T)(f)=\langle T,(\nabla^{\omega})^k f\rangle,
\end{equation}
for all $T\in C^{\infty}(P,\otimes^k\g_{-1}\otimes\Delta^{\delta}(\R^{m}))$ and $f\in C^{\infty}(P,\Delta^{\lambda}(\R^{m}))$.
\end{defi}
%%%%%
Explicitly, when the symbol $T$ writes $t A\otimes h_1\otimes\cdots\otimes h_k$ 
for $t\in C^{\infty}(P)$, $A\in\Delta^{\delta}(\R^{m})$ and $h_1,\cdots,
h_k\in\R^m\cong\g_{-1}$ then one has
\[Q_{\omega}(T)f=t A\circ
L_{\omega^{-1}(h_{1})}\circ\cdots\circ L_{\omega^{-1}(h_{k})}f,\]
where $t$ is considered as a multiplication operator.

\subsection{The map $\gamma$}

\begin{defi}\label{gamma1}
We define $\gamma$ on $\otimes^k\gm\otimes\Delta^{\delta}(\R^{m})$ by 
\[\begin{array}{r}\gamma(h)(x_1\otimes\cdots\otimes x_k\otimes A)=\lambda\sum_{i=1}^k tr([h,x_{i}]) x_1\otimes\cdots(i)\cdots\otimes x_k\otimes A\\
+\sum_{i=1}^k\sum_{j>i}x_1\otimes\cdots(i)\cdots\otimes\underbrace{[[h,x_i],x_j]}_{(j)}\otimes\cdots\otimes
x_k\otimes A.
\end{array}\]
for every $x_1,\cdots, x_k\in\gm$, $A\in\Delta^{\delta}(\R^{m})$ and $h\in \g_1$. Then we extend it to $C^\infty(P,\otimes^k\gm\otimes\Delta^{\delta}(\R^{m}))$ by $C^\infty(P)$-linearity.
\end{defi}
\begin{defi}
A trace-free symbol $S$ is a symbol such that $i(g)S=0$ if $g$ is a metric belonging to the conformal structure $P$.
\end{defi}
If $S$ is an equivariant function representing a trace-free symbol, $i(g_{0})S=0$ if $g_{0}$ represents the canonical metric on $\R^{m}$ corresponding to the conformal structure $P$ (see \cite{MR}, section 3). It is then easy to show that
\begin{prop}\label{gammafree}
If $S$ is a trace-free symbol of degree $k$, \\
$\gamma(h)S=-k(\lambda m+k-1)i(h)S$. In particular, $\gamma(h)S$ is trace-free.
\end{prop}
\subsection{Casimir-like operators}

Recall that we can define an operator called the Casimir operator $\Cf$ on $C^\infty(P,\otimes^k\gm\otimes\Delta^{\delta}(\R^{m}))$ (see \cite{MR}). This operator $\Cf$ is semi-simple. The vector space $\otimes^k\gm\otimes\Delta^{\delta}(\R^{m})$ can be decomposed as an $o(p,q)$-representation into irreducible components (since $o(p,q)$ is semi-simple):
\[\otimes^k\gm\otimes\Delta^{\delta}(\R^{m})=\oplus_{s=1}^{n_k}I_{k,s}.\]
The restriction of $\Cf$ to $C^\infty(P,I_{k,s})$ is then a scalar multiple of the identity.

We defined in \cite{MR} two other operators. If we denote respectively by $(e_1,\ldots,e_m)$ and $(\epsilon^{1},\ldots,\epsilon^{m})$ a basis of $\g_{-1}$ and a basis of $\g_{1}$ which are dual with respect to the Killing form of $so(p+1,q+1)$, then 
\begin{defi}The operator $N^{\omega}$ is defined on $C^\infty(P,\otimes^k\gm\otimes\Delta^{\delta}(\R^{m}))$ by
\[N^{\omega}=-2\sum_{i=1}^m\gamma(\varepsilon^i)L_{\omega^{-1}}(e_i),\]and we set
\[C^\omega:=\Cf+N^\omega.\]
\end{defi}
%%%%
\subsection{Construction of the quantization}\label{curvconst} 

Recall that $\otimes^k\g_{-1}\otimes\Delta^{\delta}(\R^{m})$ is decomposed as a representation of $o(p,q)$ as the direct sum of irreducible components $I_{k,s}$ with $0\leq s\leq \frac {k}{2}$ (see \cite{DLO}). Remark that if $S$ is a trace-free symbol of degree $k$, then $S\in I_{k,0}$. Denote by $E_{k,s}$ the space $C^\infty(P,I_{k,s})$ and by $\alpha_{k,s}$ the eigenvalue of $\Cf$ restricted to $E_{k,s}$.

The tree-like susbspace $\tree(I_{k,s})$ associated to $I_{k,s}$ is defined by
\[
\tree(I_{k,s})=\bigoplus_{l\in\N}\tree^l(I_{k,s}),
\]
where $\tree^0(I_{k,s})=I_{k,s}$ and
$\tree^{l+1}(I_{k,s})=\gamma(\g_1)(\tree^l(I_{k,s}))$,
for all $l\in\N$.
The space $\tree^l(E_{k,s})$ is then defined in the same way. Since $\gamma$ is $C^\infty(P)$-linear, this space is equal to $C^{\infty}(P,\tree^l(I_{k,s}))$.

\begin{defi}
A value of $\delta$ is \emph{critical} if there exists $k,s$ such that the eigenvalue $\alpha_{k,s}$ corresponding to an irreducible component $I_{k,s}$ of $\otimes^k\g_{-1}\otimes\Delta^{\delta}(\R^{m})$
belongs to the spectrum of the restriction of $C^{\flat}$ to 
$\bigoplus_{l\geq 1}\tree^l(E_{k,s})$.
\end{defi}
Recall now the following result :
\begin{thm}\label{hat}
If $\delta$ is not critical, for every $T$
in $C^{\infty}(P,I_{k,s})$, (where $I_{k,s}$ is an irreducible component of $\otimes^k\g_{-1}\otimes\Delta^{\delta}(\R^{m})$) there exists a unique function $\hat{T}$ in $C^{\infty}(P,\tree(I_{k,s}))$
 such that
\begin{equation}\label{curvP}\left\{\begin{array}{lll}
\hat{T}&=&T_k+\cdots+T_0,\quad T_k=T\\
C^{\omega}(\hat{T})&=&\alpha_{k,s}\hat{T}.
\end{array}\right.
\end{equation} 
\end{thm}
%%%%
This result allows to define the main ingredient in order to define the
quantization : The "modification map'', acting on symbols.
\begin{defi}
Suppose that $\delta$ is not critical. Then
the map 
\[R: \oplus_{k=0}^\infty C^{\infty}(P,S^k_{\delta})\to \oplus_{k=0}^\infty C^{\infty}(P,\otimes^k\gm\otimes\Delta^{\delta}(\R^{m}))\]
is the linear extension of the association $T\mapsto \hat{T}$.
\end{defi}
And finally, the main result :
\begin{thm}\label{quantif}
If $\delta$ is not critical, then the formula
\[Q_M: (g,T)\mapsto Q_M(g,T)(f)=(p^*)^{-1}[Q_{\omega}(R(p^*T))(p^*f)],\]
(where $Q_{\omega}$ is given by (\ref{Qom})) defines a natural and conformally invariant quantization.
\end{thm}
\section{The first explicit formula}
Define now the numbers $\gamma_{2k-l}$ :
\[\gamma_{2k-l} = \frac{m+2 k - l -m\delta}{m}.\]
We will say that a value of $\delta$ is \it{critical} \rm{if there are} $k,l\in \N$ such that
$2\leq l\leq k+1$ and $\gamma_{2k-l}=0$.

We can then give the formula giving the natural and conformally invariant quantization in terms of the normal Cartan connection for the trace-free symbols (see \cite{moi} for the definitions of $\nabla_{s}^{\omega}$ and $Div^{\omega}$) :
\begin{thm} If $\delta$ is not critical, then the collection of maps \\
$Q_M : S^{2}T^{*}M\times \mathcal{S}_{\delta}^{k}(M)\to
\mathcal{D}_{\lambda,\mu}(M)$ defined by
\begin{equation}\label{formula}Q_M(g,S)(f) = p^{*^{-1}}(\sum_{l=0}^k C_{k,l} \langle Div^{\omega^l}
p^*S,\nabla_s^{\omega^{k-l}}p^*f\rangle)\end{equation}
defines a conformally invariant natural quantization for the trace-free symbols if
\[C_{k,l} =\frac{(\lambda + \frac{k-1}{m})\cdots (\lambda +
  \frac{k-l}{m})}{\gamma_{2k-2}\cdots
  \gamma_{2k-l-1}}\left(\begin{array}{c}k\\l\end{array}\right),\forall l\geq 1,\quad C_{k,0}=1.\]
\end{thm}
\begin{proof}
Thanks to Theorem \ref{hat}, to the definition of $N^{\omega}$ and to Proposition \ref{gammafree}, one has
$$S_{l}=\frac {2\sum_{i=1}^m \gamma(\varepsilon^i)L_{\omega^{-1}}(e_i)S_{l+1}}{\alpha_{l,0}-\alpha_{k,0}},\quad 0\leq l\leq k-1.$$
One concludes using Proposition \ref{gammafree} and the fact that (see \cite{DLO}) :
$$\alpha_{k,0}=2k(1-k+m(\delta-1))-m^{2}\delta(\delta-1).$$
Indeed, if $(e_1,\ldots,e_m)$ and $(\epsilon^{1},\ldots,\epsilon^{m})$ denote respectively the canonical bases of $\R^{m}$ and $\R^{m*}$, $(e_1,\ldots,e_m)$ and $(-\epsilon^{1},\ldots,-\epsilon^{m})$ are Killing-dual with respect to the Killing form given in \cite{DLO}. One applies eventually Theorem \ref{quantif}.
\end{proof}
\section{The second explicit formula}
In order to obtain an explicit formula for the quantization, we need to know the
developments of the operators $\nabla^{\omega^{l}}$ and $Div^{\omega^{l}}$ in terms of
operators on $M$.

Let $\gamma$ be a connection on $P_{0}$ corresponding to a covariant derivative $\nabla$ and belonging to the underlying structure of a conformal structure $P$. Recall that $\gamma$ is the Levi-Civita connection of a metric belonging to $P$. We denote by $\tau$ the corresponding function on $P$ with values in $\g_{1}$, by $\Gamma$ the corresponding deformation tensor (see \cite{Slo1}) and by $\omega$ the normal Cartan connection on $P$.

Let $(V,\rho)$ be a representation of $G_{0}$ inducing a representation
$(V,\rho_{*})$ of $\g_{0}$. If we denote by $\rho_{*}^{(l)}$ the canonical
representation on $\otimes^{l}\g_{-1}^{*}\otimes V$ and if $s\in C^{\infty}(P_{0},V)_{G_{0}}$, then the development of $\nabla^{\omega^{l}}(p^{*}s)(X_1,\ldots,X_{l})$ is obtained inductively
as follows (see \cite{Slo1}, \cite{moi}) :

\vspace{0.2cm}$\nabla^{\omega^{l}}(p^{*}s)(X_1,\ldots,X_l)=\rho_{*}^{(l-1)}([X_{l},\tau])(\nabla^{\omega^{l-1}}(p^{*}s))(X_1,\ldots,X_{l-1})$

\vspace{0.2cm}\hspace{2cm}$+S_{\tau}(\nabla^{\omega^{l-1}}(p^{*}s))(X_1,\ldots,X_{l-1})$

\vspace{0.2cm}\hspace{2cm}$+S_{\nabla}(\nabla^{\omega^{l-1}}(p^{*}s))(X_1,\ldots,X_{l-1})$

\vspace{0.2cm}\hspace{2cm}$+S_{\Gamma}(\nabla^{\omega^{l-1}}(p^{*}s))(X_1,\ldots,X_{l-1}).$

\vspace{0.2cm}Recall that $S_{\tau}$ replaces successively each $\tau$ by $-\frac {1}{2}[\tau, [\tau, X_{l}]]$, that $S_{\nabla}$ adds successively a covariant derivative on the covariant derivatives of $\Gamma$ and $s$ and that $S_{\Gamma}$ replaces successively each $\tau$ by $\Gamma.X_{l}$.

Recall too that $\Gamma$ is equal in the conformal case to (see \cite{Slo1}) :
$$\frac {-1}{m-2}(\mathrm{Ric}-\frac {g_{0}\mathrm{R}}{2(m-1)}),$$
where $\mathrm{Ric}$ and $\mathrm{R}$ denote the equivariant functions on $P$ representing respectively the Ricci tensor and the scalar curvature of the connection $\gamma$.
\begin{prop}\label{nablao}
If $f\in C^{\infty}(P_{0},\Delta^{\lambda}(\R^{m}))_{G_{0}}$, then
$\nabla^{\omega^{l}}(p^{*}f)(X,\ldots,X)$ is equal to $g_{0}(X,X)T(X,\ldots,X)$, where $T\in C^{\infty}(P,\otimes^{l-2}\R^{m*}\otimes\Delta^{\lambda}(\R^{m}))$, plus a linear combination of terms of the form
$$(\otimes^{n_{-1}}\tau\otimes p^{*}(\otimes^{n_{l-2}}\nabla^{l-2}\Gamma\otimes\ldots\otimes\otimes^{n_0}\Gamma\otimes\nabla^{q}f))(X,\ldots,X).$$
If we denote by $T(n_{-1},\ldots,n_{l-2},q)$ such a term, then $\nabla^{\omega^{l+1}}(p^{*}f)(X,\ldots,X)$ is equal to the corresponding linear combination of the following sums
$$(-\lambda m-2l+n_{-1})T(n_{-1}+1,\ldots,n_{l-2},q)+T(n_{-1},\ldots,n_{l-2},q+1)$$
$$+\sum_{j=-1}^{l-2}n_{j}T(n_{-1},\ldots,n_{j}-1,n_{j+1}+1,\ldots,n_{l-2},q)$$
plus $g_{0}(X,X)T'(X,\ldots,X)$, where $T'\in C^{\infty}(P,\otimes^{l-1}\R^{m*}\otimes\Delta^{\lambda}(\R^{m}))$. 
\end{prop}
\begin{proof}
The proof is similar to the proof of Proposition 7 in \cite{moi}.
\end{proof}
One deduces easily from Proposition \ref{nablao} the following corollary (see \cite{moi} for the definition of $\nabla_{s}$) :
\begin{prop}\label{devnabla}
If $f\in C^{\infty}(P_{0},\Delta^{\lambda}(\R^{m}))_{G_{0}}$, then
$\nabla_{s}^{\omega^{l}}(p^{*}f)$ is equal to $g_{0}\vee T$, where $T\in C^{\infty}(P,S^{l-2}\R^{m*}\otimes\Delta^{\lambda}(\R^{m}))$, plus a linear combination of terms of the form
$$(\tau^{n_{-1}}\vee p^{*}((\nabla^{l-2}\Gamma)^{n_{l-2}}\vee\ldots\vee\Gamma^{n_0}\vee\nabla_{s}^{q}f)).$$
If we denote by $T(n_{-1},\ldots,n_{l-2},q)$ such a term, then $\nabla_{s}^{\omega^{l+1}}(p^{*}f)$ is equal to the corresponding linear combination of the following sums
$$(-\lambda m-2l+n_{-1})T(n_{-1}+1,\ldots,n_{l-2},q)+T(n_{-1},\ldots,n_{l-2},q+1)$$
$$+\sum_{j=-1}^{l-2}n_{j}T(n_{-1},\ldots,n_{j}-1,n_{j+1}+1,\ldots,n_{l-2},q)$$
plus $g_{0}\vee T'$, where $T'\in C^{\infty}(P,S^{l-1}\R^{m*}\otimes\Delta^{\lambda}(\R^{m}))$. 
 \end{prop}
\begin{proof}
The proof is similar to the proof of Proposition 8 in \cite{moi}.
\end{proof}
Remark that the action of the algorithm on the generic term of the part of the development
of $\nabla_s^{\omega^{l}}(p^{*}f)$ that does not contain factors $g_{0}$ can be summarized. Indeed, this action
gives first
$$(-\lambda m-2l+n_{-1})T(n_{-1}+1,\ldots,n_{l-2},q).$$
It gives next
$$n_{-1}T(n_{-1}-1,n_{0}+1,\ldots,n_{l-2},q).$$
Finally, it makes act the covariant derivative $\nabla_s$ on
$$(\nabla_s^{l-2}\Gamma)^{n_{l-2}}\vee\ldots\vee\Gamma^{n_0}\vee\nabla_s^{q}f.$$
From now, we will denote by $r$ the following multiple of the tensor Ric (recall that Ric is symmetric for a metric connection) :
$$r:=\frac {1}{(2-m)}\mathrm{Ric}.$$
In the following proposition, $Div$ denotes the divergence operator :
\begin{prop}\label{Div}
If $S\in C^{\infty}(P_{0},\Delta^{\delta}\R^{m}\otimes
S^{k}\R^{m})_{G_0}$ is trace-free, then $Div^{\omega^l}(p^{*}S)$ is a
linear combination of terms of the form
$$\langle\tau^{n_{-1}}\vee p^{*}((\nabla_s^{k-2}r)^{n_{k-2}}\vee\ldots\vee
r^{n_0}), p^{*}(Div^{q}S)\rangle.$$
If we denote by $T(n_{-1},\ldots,n_{l-2},q)$ such a term, then
$Div^{\omega}T(n_{-1},\ldots,n_{l-2},q)$ is equal to
$$(\gamma_{2(k-l)-2}m+n_{-1})T(n_{-1}+1,\ldots,n_{l-2},q)+T(n_{-1},\ldots,n_{l-2},q+1)$$
$$+\sum_{j=-1}^{l-2}n_{j}T(n_{-1},\ldots,n_{j}-1,n_{j+1}+1,\ldots,n_{l-2},q).$$
\end{prop}
\begin{proof}
The proof is exactly similar to the one of Proposition 9 in \cite{moi}, using the fact that $S$ and its divergences are trace-free.
\end{proof}
Remark that the action of the algorithm on the generic term of the development
of $Div^{\omega^{l}}(p^{*}S)$ can be summarized. Indeed, this action gives
first
$$(\gamma_{2(k-l)-2}m+n_{-1})T(n_{-1}+1,\ldots,n_{l-2},q).$$
It gives next
$$n_{-1}T(n_{-1}-1,n_{0}+1,\ldots,n_{l-2},q).$$
 Finally, it makes act the divergence $Div$ on
$$\langle(\nabla_s^{k-2}r)^{n_{k-2}}\vee\ldots\vee
r^{n_0}, Div^{q}S\rangle.$$

Because of the previous propositions, the quantization can be written as a
linear combination of terms of the form
$$\langle\langle\tau^{n_{-1}}\vee p^{*}((\nabla_s^{k-2}r)^{n_{k-2}}\vee\ldots\vee
r^{n_0}), p^{*}(Div^{q}S)\rangle, p^{*}(\nabla_s^{l}f)\rangle.$$
In this expression, recall that it suffices to consider the terms for which
$n_{-1}=0$ (see \cite{moi}).

In the sequel, we will need two operators that we will call $T_{1}$ and
$T_{2}$.

If $T$ is a tensor of type $\left(\begin{array}{c}0\\j\end{array}\right)$ with values
in the $\lambda$-densities, then
$$T_{1}T=(-\lambda m-j)(j+1)\Gamma\vee T.$$

If $S$ is a trace-free symbol of degree $j$, then
$$T_{2}S=(m\gamma_{2k-2}-k+j)(k-j+1)i(r)S.$$

The following results give the explicit developments of
$\nabla_s^{\omega^{l}}(p^{*}f)$ and of $Div^{\omega^{l}}(p^{*}S)$ :
\begin{prop}
The term of degree $t$ in $\tau$ in the part of the development of
$\nabla_s^{\omega^{l}}(p^{*}f)$ that does not contain factors $g_{0}$ is equal to
$$\left(\begin{array}{c}l\\t\end{array}\right)\prod_{j=1}^{t}(-\lambda m-l+j)p^{*}(\pi_{l-t}(\sum_{j=0}^{l-t}(\nabla_s+T_{1})^{j})f),$$
where $\pi_{l-t}$ denotes the projection on the operators of degree $l-t$ (the degree of $\nabla_{s}$ is $1$ whereas
the degree of $T_{1}$ is $2$). We set $\prod_{j=1}^{t}(-\lambda m-l+j)$ equal to $1$ if $t=0$.
\end{prop}
\begin{proof}
The proof is exactly similar to the one of Proposition 10 in \cite{moi}.
\end{proof}

\begin{prop}
If $S$ is trace-free, the term of degree $t$ in $\tau$ in the development of
$Div^{\omega^{l}}(p^{*}S)$ is equal to
$$\left(\begin{array}{c}l\\t\end{array}\right)\prod_{j=1}^{t}(\gamma_{2k-2}m-l+j)p^{*}(\pi_{t-l}(\sum_{j=0}^{l-t}(Div+T_{2})^{j})S),$$
where $\pi_{t-l}$ denotes the projection on the operators of degree $t-l$ (the degree of $Div$ is $-1$ whereas the degree
of $T_2$ is $-2$). We set the product \\
$\prod_{j=1}^{t}(\gamma_{2k-2}m-l+j)$ equal to $1$ if $t=0$.
\end{prop}
\begin{proof}
The proof is completely similar to the one of Proposition 11 in \cite{moi}.
\end{proof}

We can now write the explicit formula giving the natural and conformally invariant quantization for the trace-free symbols :
\begin{thm}
The quantization $Q_{M}$ for the trace-free symbols is given by the following formula :
$$Q_{M}(g,S)(f)=\sum_{l=0}^{k}C_{k,l}\langle\pi_{l}(\sum_{j=0}^{l}(Div+T_{2})^{j})S,
\pi_{k-l}(\sum_{j=0}^{k-l}(\nabla_s+T_{1})^{j})f\rangle.$$
\end{thm}
Remark that as $S$ and its divergences are trace-free, one can replace in the definition of the operators $T_{1}$ the deformation tensor $\Gamma$ by $r$.
\vspace{0.2cm}One can easily derive from this formula the formula at the third order. Indeed, if we denote by $D$, $T$, $\partial
T$ the operators $\nabla_{s}$, $r\vee$ and $(\nabla_{s}r)\vee$ (resp. $Div$,
$i(r)$ and $i(\nabla_{s}r)$) and if we denote by $\beta$ the number
$-\lambda m$ (resp. $\gamma_{4}m$), one obtains :
$$\pi_{1}(\sum_{j=0}^{1}(D+T)^{j})=D,\quad
\pi_{2}(\sum_{j=0}^{2}(D+T)^{j})=D^{2}+\beta T,$$
$$\pi_{3}(\sum_{j=0}^{3}(D+T)^{j})=D^{3}+\beta
DT+2(\beta-1)TD=D^{3}+(3\beta-2)TD+\beta(\partial T).$$
We can then write the formula at the third order :
$$\langle
S,(\nabla_{s}^{3}-(3m\lambda+2)r\vee\nabla_{s}-\lambda m(\nabla_{s}r))f\rangle$$
$$+C_{3,1}\langle
Div S,(\nabla_{s}^{2}-m\lambda r)f\rangle+C_{3,2}\langle(Div^{2}+m\gamma_{4}i(r))S,
\nabla_{s}f\rangle$$
$$+C_{3,3}\langle(Div^{3}+(3\gamma_{4}m-2)i(r)Div+m\gamma_{4}i(\nabla_{s}r))S,f\rangle.$$
At the second order, the formula is simply :
$$\langle
S,(\nabla_{s}^{2}-m\lambda r)f\rangle+C_{2,1}\langle
Div S,\nabla_{s}f\rangle+C_{2,2}\langle(Div^{2}+m\gamma_{2}i(r))S,f\rangle.$$
\section{Acknowledgements}
It is a pleasure to thank P. Mathonet and V. Ovsienko for numerous fruitful discussions and for their interest in our work. We thank the Belgian FNRS for his Research Fellowship.

\end{document}